\documentclass[twoside,reqno]{amsart}

\usepackage{amsmath,amssymb,amsbsy,amstext,amsxtra,latexsym}
\usepackage{graphicx}
\usepackage{subfig}
\usepackage{enumerate}
\usepackage[all]{xy}

\usepackage{ifpdf}
\ifpdf
\usepackage[colorlinks=true,pdftex]{hyperref}
\else
\usepackage[colorlinks=true,dvipdfm]{hyperref}
\fi

\hypersetup{%
citecolor=black,%
filecolor=black,%
linkcolor=black,%
urlcolor=black,%
}

\newtheorem{theorem}{Theorem}[section]

\newtheorem{lemma}[theorem]{Lemma}
\newtheorem{proposition}[theorem]{Proposition}

\theoremstyle{definition}

\newtheorem*{remarks}{Remarks}

\numberwithin{equation}{section}

\newcommand{\uc}[1]{\ensuremath \overset{#1}{\circ}}

\newcommand{\blup}[2]{\ensuremath #1 \sharp #2 \overline{\mathbb{CP}^2}}

\DeclareMathOperator{\divisor}{Div}

\def\sheaf#1{\ensuremath \mathcal#1}

\begin{document}

\title[A surface of general type with $p_g=0$, $K^2=4$, and $\pi_1=\mathbb{Z}/2\mathbb{Z}$]{A complex surface of general type \\ with $p_g=0$, $K^2=4$, and $\pi_1=\mathbb{Z}/2\mathbb{Z}$}

\author{Heesang Park}

\address{Department of Mathematical Sciences, Seoul National University, 599 Gwanak-ro, Gwanak-gu, Seoul 151-747, Korea}

\email{hspark@math.snu.ac.kr}

\date{October 17, 2009; revised at November 2, 2009}

\subjclass[2000]{Primary 14J29; Secondary 14J10, 14J17, 53D05}

\keywords{$\mathbb{Q}$-Gorenstein smoothing, rational blow-down surgery, surface of general type}

\begin{abstract}
We construct a minimal complex surface of general type with $p_g=0$, $K^2 =4$, and $\pi_1=\mathbb{Z}/2\mathbb{Z}$ using a rational blow-down surgery and a $\mathbb{Q}$-Gorenstein smoothing theory. In a similar fashion, we also construct a symplectic $4$-manifold with $b_2^+=1$, $K^2=5$, and $\pi_1=\mathbb{Z}/2\mathbb{Z}$.
\end{abstract}

\maketitle

\section{Introduction}

In this paper we construct a new minimal complex surface of general type with $p_g=0$ and $K^2=4$. It is a fundamental problem in the classification of complex surfaces to find a minimal complex surface of general type with $p_g=0$. Recently simply connected complex surfaces of general type with $p_g=0$ and $K^2 \le 4$ are constructed; Y. Lee and J. Park~\cite{Lee-Park-K^2=2}, the author-J. Park-D. Shin~\cite{PPS-K3, PPS-K4}. Also many families of non-simply connected complex surfaces of general type with $p_g=0$ have been constructed; cf.~BHPV~\cite[VII]{BHPV}.

It is nevertheless an intriguing problem to find complex surfaces of general type with $p_g=0$ and small nonzero fundamental groups, especially, $\pi_1 = \mathbb{Z}/2\mathbb{Z}$ (the smallest nonzero group), because there are no known examples with $K^2 \ge 4$ and $\pi_1 = \mathbb{Z}/2\mathbb{Z}$. For instance, the first example with $p_g=0$, $K^2=1$, and $\pi_1 = \mathbb{Z}/2\mathbb{Z}$ was constructed by Barlow~\cite{Barlow}. It is very recent that examples with $p_g=0$, $K^2 > 1$, and $\pi_1 = \mathbb{Z}/2\mathbb{Z}$ are constructed: D. Cartwright and T. Steger~\cite{Cartwright-Steger} constructed complex surfaces of general type with $p_g=0$, $K^2=3$, and $\pi_1 = \mathbb{Z}/2\mathbb{Z}$. J. Keum and Y. Lee~\cite{Keum-Lee} constructed complex surfaces of general type with $p_g=0$, $K^2=1,2,3$, and $\pi_1=\mathbb{Z}/2\mathbb{Z}$. However it is not known yet whether there are complex surfaces of general type with $p_g=0$, $K^2=4$, and $\pi_1 = \mathbb{Z}/2\mathbb{Z}$.

Motivated by the work of J. Keum and Y. Lee~\cite{Keum-Lee}, we extend their result to the $K^2=4$ case in this paper. The main result of this paper is the following theorem.

\begin{theorem}
\label{thm-main}
There exists a minimal complex surface of general type with $p_g=0$, $K^2=4$, and $\pi_1=\mathbb{Z}/2\mathbb{Z}$.
\end{theorem}

The key ingredient of this paper is to construct a surface $Z$ with special configurations of rational curves by appropriately blowing-up several times starting with an Enriques surface. And then we apply a similar method developed in Y. Lee and J. Park~\cite{Lee-Park-K^2=2} to the surface $Z$. That is, we contract the special chains of $\mathbb{CP}^1$'s from the surface $Z$ so that we get a surface $X$ with permissible singular points. We prove that there is a global $\mathbb{Q}$-Gorenstein smoothing of the singular surface $X$. For this we show that the obstruction to a global $\mathbb{Q}$-Gorenstein smoothing is zero by using a similar strategy as in J. Keum and Y. Lee~\cite{Keum-Lee}. Then a general fiber of a global $\mathbb{Q}$-Gorenstein smoothing of the singular surface $X$ is a complex surface of general type with $p_g=0$ and $K^2=4$. Finally we show that the surface has $\pi_1 = \mathbb{Z}/2\mathbb{Z}$ by applying a Milnor fiber theory and a rational blow-down surgery.

We also construct a symplectic $4$-manifold with $b_2^+=1$, $K^2=5$, and $\pi_1=\mathbb{Z}/2\mathbb{Z}$ using the same technique.

\begin{theorem}
\label{thm-main2}
There exists a symplectic $4$-manifold with $b_2^+=1$, $K^2 =5$, and $\pi_1 = \mathbb{Z}/2\mathbb{Z}$.
\end{theorem}

\subsubsection*{Acknowledgements}

The author would like to thank Professor Jongil Park for encouraging and guiding me during the course of this work. The author would also like to thank Professor Yongnam Lee for kind explanation of his paper, Keum-Lee~\cite{Keum-Lee} and some valuable comments on the first draft of this article, and Professor Dongsoo Shin for valuable discussions and drawing figures in this paper.

\section{Main construction}\label{section:Construction}

We start with a special elliptic fibration on an Enriques surface. According to Kondo~\cite{Kondo}, there is an Enriques surface $Y$ with an elliptic fibration over $\mathbb{P}^1$. In particular the Enriques surface $Y$ has an $I_9$-singular fiber and a nodal singular fiber $F$ which are not multiple fibers and two bisections $S_1$ and $S_2$ not passing through the node of the nodal fiber $F$. The configuration of the fibers and sections are as in Figure~\ref{figure:Y}.

\begin{figure}[hbtb]
\centering
\includegraphics{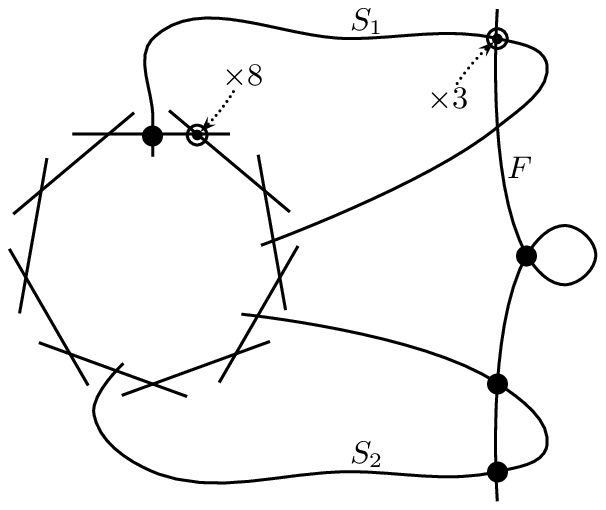}
\caption{An Enriques surface $Y$}
\label{figure:Y}
\end{figure}

We blow up four times totally at the four marked points $\bullet$. We blow up again three times and eight times at the two marked points $\bigodot$, respectively. We then get a surface $Z = \blup{Y}{15}$; Figure~\ref{figure:Z}. There exist two disjoint linear chains of $\mathbb{CP}^1$'s in $Z$:
\begin{align*}
&C_{19,13}: \uc{-2}-\uc{-2}-\uc{-9}-\uc{-2}-\uc{-2}-\uc{-2}-\uc{-2}-\uc{-4} \\
&C_{73,50}: \uc{-2}-\uc{-2}-\uc{-7}-\uc{-6}-\uc{-2}-\uc{-3}-\uc{-2}-\uc{-2}-\uc{-2}-\uc{-2}-\uc{-4}
\end{align*}

\begin{figure}[hbtb]
\centering
\includegraphics{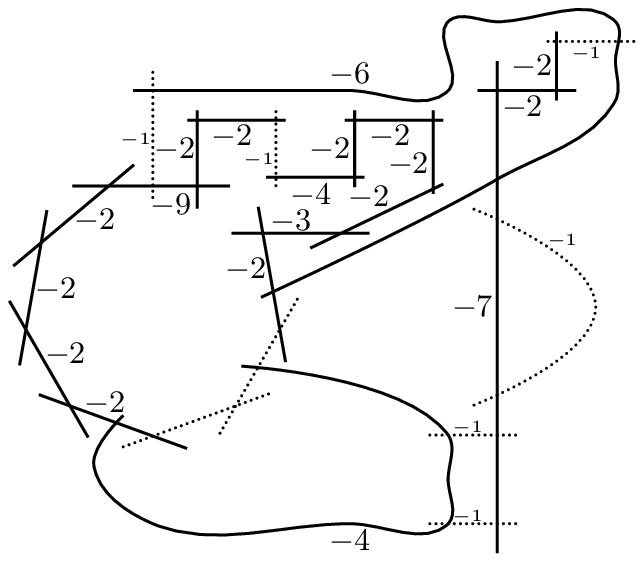}
\caption{A surface $Z = \blup{Y}{15}$}
\label{figure:Z}
\end{figure}

By applying $\mathbb{Q}$-Gorenstein smoothing theory to the surface  $Z$ as in~\cite{Lee-Park-K^2=2, PPS-K3, PPS-K4}, we construct a complex surface of general type with $p_g=0$ and $K^2=4$. That is, we first contract the two chains of $\mathbb{CP}^1$'s from the surface $Z$ so that it produces a normal projective surface $X$ with two permissible singular points. In Section~\ref{section:Q-Gorenstein} we will show that the singular surface $X$ has a global $\mathbb{Q}$-Gorenstein smoothing. Let $X_t$ be a general fiber of the $\mathbb{Q}$-Gorenstein smoothing of $X$. Since $X$ is a singular surface with $p_g=0$ and $K^2=4$, by applying general results of complex surface theory and a $\mathbb{Q}$-Gorenstein smoothing theory, one may conclude that a general fiber $X_t$ is a complex surface with $p_g=0$ and $K^2=4$. Furthermore, it is not difficult to show that a general fiber $X_t$ is minimal by using a similar technique in~\cite{Lee-Park-K^2=2, PPS-K3, PPS-K4}; hence $X_t$ is of general type. Finally it remains to show that $\pi_1(X_t)=\mathbb{Z}/2\mathbb{Z}$.

\begin{proposition}\label{proposition:pi_1=Z_2}
$\pi_1(X_t)=\mathbb{Z}/2\mathbb{Z}$.
\end{proposition}

\begin{proof}
Let $Z_{73, 19}$ be a rational blow-down $4$-manifold obtained from $Z$ by replacing two disjoint configurations $C_{73,50}$ and $C_{19,13}$ with the corresponding rational balls $B_{73,50}$ and $B_{19,13}$, respectively. Then, since a general fiber $X_t$ of a $\mathbb{Q}$-Gorenstein smoothing of $X$ is diffeomorphic to the rational blow-down $4$-manifold $Z_{73,19}$ by a Milnor fiber theory, we have $\pi_1(X_t)=\pi_1(Z_{73,19})$. Hence it suffices to show that $\pi_1(Z_{73,19}) = \mathbb{Z}_2$.

We first decompose the surface $Z$ into
\begin{equation*}
Z = Z_0 \cup \{ C_{73,50} \cup C_{19,13}\}.
\end{equation*}
Then the $4$-manifold $Z_{73,19}$ can be decomposed into
\begin{equation*}
Z_{73,19} = Z_0 \cup \{B_{73,50} \cup B_{19,13}\}.
\end{equation*}
Let $\alpha$ and $\beta$ be normal circles of disk bundles $C_{73,50}$ and $C_{19,13}$ over the $(-2)$-curve and the $(-4)$-curve, respectively; Figure~\ref{figure:Z-normal-circles}.

\begin{figure}[hbtb]
\centering
\includegraphics{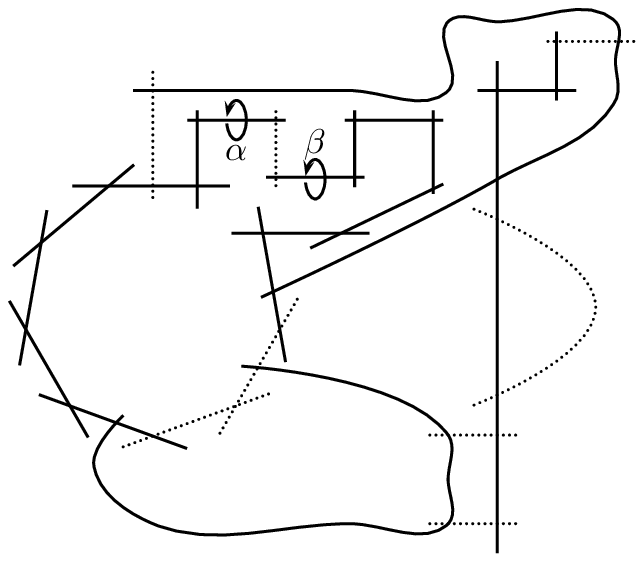}
\caption{Two normal circles on a surface $Z$}
\label{figure:Z-normal-circles}
\end{figure}

Note that $\pi_1(Z) = \mathbb{Z}_2$ for $\pi_1(Y) = \mathbb{Z}_2$. Since $C_{73,50}$ and $C_{19,13}$ is simply connected, by Van Kampen theorem, we have
\begin{equation*}
\mathbb{Z}_2 = \pi_1(Z) = \pi_1(Z_0)/\langle N_{\gamma_1 i_{\ast}(\alpha) \gamma_1^{-1}}, N_{\gamma_2 j_{\ast}(\beta) \gamma_2^{-1}}\rangle,
\end{equation*}
where $i_{\ast} : \mathbb{Z}_{73^2} \to \pi_1(Z_0)$ and $j_{\ast}: \mathbb{Z}_{19^2} \to \pi_1(Z_0)$ are the induced homomorphisms by the inclusions $i : C_{73,50} \cap Z_0 \to Z_0$ and $j: C_{19,13} \cap Z_0 \to Z_0$, respectively, and $\gamma_1$ (or $\gamma_2$) is a path connecting $\alpha$ (or $\beta$) and the reference point, respectively. Since two normal circles $\alpha$ and $\beta$ lie on the $(-1)$-sphere as in Figure~\ref{figure:Z-normal-circles}, $\gamma_1 i_{\ast}(\alpha) \gamma_1^{-1}$ and $\gamma_2 j_{\ast}(\beta) \gamma_2^{-1}$ have the same order in $\pi_1(Z_0)$. However, since two integers $73$ and $19$ are relatively prime, we have $\langle N_{\gamma_1 i_{\ast}(\alpha) \gamma_1^{-1}}, N_{\gamma_2 j_{\ast}(\beta) \gamma_2^{-1}}\rangle = 1$. Therefore
\begin{equation*}
\pi_1(Z_0) = \mathbb{Z}_2.
\end{equation*}

We now consider $\pi_1(Z_0 \cup B_{73,50})$. Note that the map $\pi_1(\partial B_{73,50}) \to \pi_1(Z_0)$ is given by
\begin{align*}
\mathbb{Z}_{73^2} \to \mathbb{Z}_2, \quad \overline{1} \mapsto \overline{0}
\end{align*}
and the map $\pi_1(\partial B_{73,50}) \to \pi_1(B_{73,50})$ is given by
\begin{equation*}
\mathbb{Z}_{73^2} \to \mathbb{Z}_{73}, \quad \overline{1} \mapsto \overline{1}.
\end{equation*}
Therefore we have
\begin{equation*}
\pi_1(Z_0 \cup B_{73,50}) = \pi_1(Z_0) \underset{\pi_1(\partial B_{73,50})}{\ast} \pi_1(B_{73,50}) = \mathbb{Z}_2 \underset{\mathbb{Z}_{73^2}}{\ast} \mathbb{Z}_{73} = \mathbb{Z}_2.
\end{equation*}
Similarly, we can conclude that
\begin{equation*}
\pi_1(Z_{73,19}) = \pi_1(Z_0 \cup B_{73,50}) \underset{\pi_1(\partial B_{19,13})}{\ast} \pi_1(B_{19,13}) = \mathbb{Z}_2 \underset{\mathbb{Z}_{19^2}}{\ast} \mathbb{Z}_{19} = \mathbb{Z}_2. \qedhere
\end{equation*}
\end{proof}

\section{Existence of a global $\mathbb{Q}$-Gorenstein smoothing}\label{section:Q-Gorenstein}

This section is devoted to a proof of the following theorem.

\begin{theorem}\label{theorem:obstruction=0}
The singular surface $X$ has a global $\mathbb{Q}$-Gorenstein smoothing.
\end{theorem}

The following proposition tells us a sufficient condition for the existence of a global $\mathbb{Q}$-Gorenstein smoothing of $X$.

\begin{proposition}[Y. Lee and J. Park~\cite{Lee-Park-K^2=2}]\label{proposition:sufficient_H^2=0}
Let $X$ be a normal projective surface with singularities of class $T$. Let $\pi : V \to X$ be the minimal resolution and let $A$ be the reduced exceptional divisor. Suppose that $H^2(T_V(-\log{A}))=0$. Then there is a global $\mathbb{Q}$-Gorenstein smoothing of $X$.
\end{proposition}

Since the contraction map $Z \to X$ is the minimal resolution of the singular surface $X$, the existence of a global $\mathbb{Q}$-Gorenstein smoothing of $X$ follows from the vanishing of the cohomology $H^2(T_Z(-\log{A}))$, where $A$ is the divisor on $Z$ consisting of the two linear chains of $\mathbb{CP}^1$'s contracted to the two singular points of $X$. On the one hand we have the following well-known result.

\begin{proposition} [{Flenner and Zaidenberg~\cite[\S1]{Flenner-Zaidenberg}}]\label{proposition:Flenner-Zaidenberg}
Let $V$ be a nonsingular surface and let $A$ be a simple normal crossing divisor in $V$. Let $f : V' \to V$ be a blowing up of $V$ at a point p of $A$. Set $A'=f^{-1}(A)_{red}$. Then $h^2(T_{V'}(-\log{A'}))=h^2(T_V(-\log{A}))$.
\end{proposition}

Let $\tau : V \to Y$ be the blowing-up at the node of the nodal singular fiber $F$ and let $E$ be the exceptional divisor of $\tau$. We denote again by $F$ the proper transforms of the nodal singular fibers $F$ on $V$. Let $D_1, \dotsc, D_7$ be a part of the $I_9$-singular fiber; Figure~\ref{figure:V}. Let
\begin{equation*}
D = D_1 + \dotsb + D_7 + S_1 + S_2 + F \in \divisor(V).
\end{equation*}
By Proposition~\ref{proposition:Flenner-Zaidenberg}, we have
\begin{equation}\label{equation:H^2(Z)=H^2(V)}
h^2(T_Z(-\log{A}))=h^2(T_V(-\log{D})).
\end{equation}

\begin{figure}[hbtb]
\centering
\includegraphics{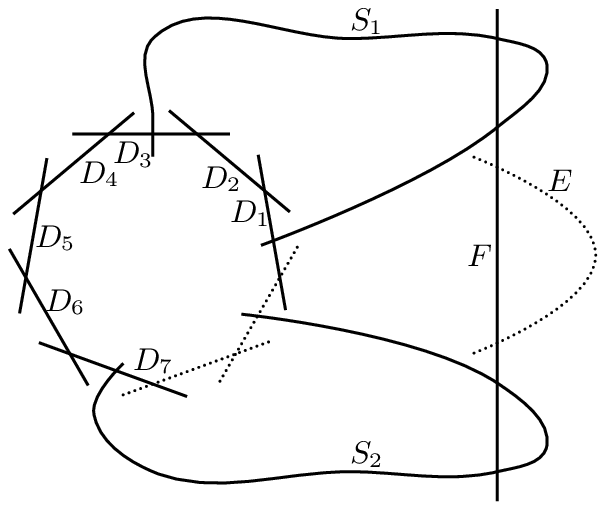}
\caption{A surface $V = \blup{Y}{}$}
\label{figure:V}
\end{figure}

Theorem~\ref{theorem:obstruction=0} follows from \eqref{equation:H^2(Z)=H^2(V)} and the following proposition:

\begin{proposition}\label{proposition:H^2=0}
$H^2(T_V(-\log{D})) = H^0(\Omega_V(\log{D})(K_V)) = 0$.
\end{proposition}

In order to prove Proposition~\ref{proposition:H^2=0}, we follow the same strategy as in J. Keum and Y. Lee~\cite{Keum-Lee}. That is, we consider a K3 surface $\overline{Y}$ blown-up two times which is a double covering of the surface $V=\blup{Y}{}$ and then we use the push-forward map of the double covering for proving that $H^0(\Omega_V(\log{D})(K_V)) = 0$.

We first construct a double covering of the surface $V=\blup{Y}{}$. According to Kondo~\cite{Kondo}, there is an unramified double covering $\phi : \overline{Y} \to Y$ from a K3 surface $\overline{Y}$ to the Enriques surface $Y$. The K3 surface $\overline{Y}$ has two $I_9$-singular fiber, two nodal singular fiber $\overline{F}_1$ and $\overline{F}_2$, and four sections $\overline{S}_1, \dotsc, \overline{S}_4$ such that $\phi(\overline{F}_1) = \phi(\overline{F}_2) = F$, $\phi(\overline{S}_1) = \phi(\overline{S}_3) = S_1$, and $\phi(\overline{S}_2) = \phi(\overline{S}_4) = S_2$; Figure~\ref{figure:Ybar}. Let $\overline{\tau} : \overline{V} \to \overline{Y}$ be the blowing-up at the two nodes of the two nodal singular fibers $\overline{F}_1$, $\overline{F}_2$ and let $\overline{E}_1$, $\overline{E}_2$ be the exceptional divisors of $\overline{\tau}$. We denote again by $\overline{F}_1$, $\overline{F}_2$ the proper transforms of the nodal singular fibers $\overline{F}_1$, $\overline{F}_2$ on $\overline{V}$; Figure~\ref{figure:Vbar}. It is clear that there is an induced unramified double covering $\psi: \overline{V} \to V$. We denote
\begin{equation*}
\Delta = \overline{D}_1 + \dotsb + \overline{D}_7 + \overline{S}_1 + \dotsb + \overline{S}_4 + \overline{F}_1 \in \divisor(\overline{V}).
\end{equation*}
Note that $\Delta \le \psi^{\ast}{D}$ and $\psi_{\ast}{\Delta} = D$.

\begin{figure}[hbtb]
\centering
\includegraphics{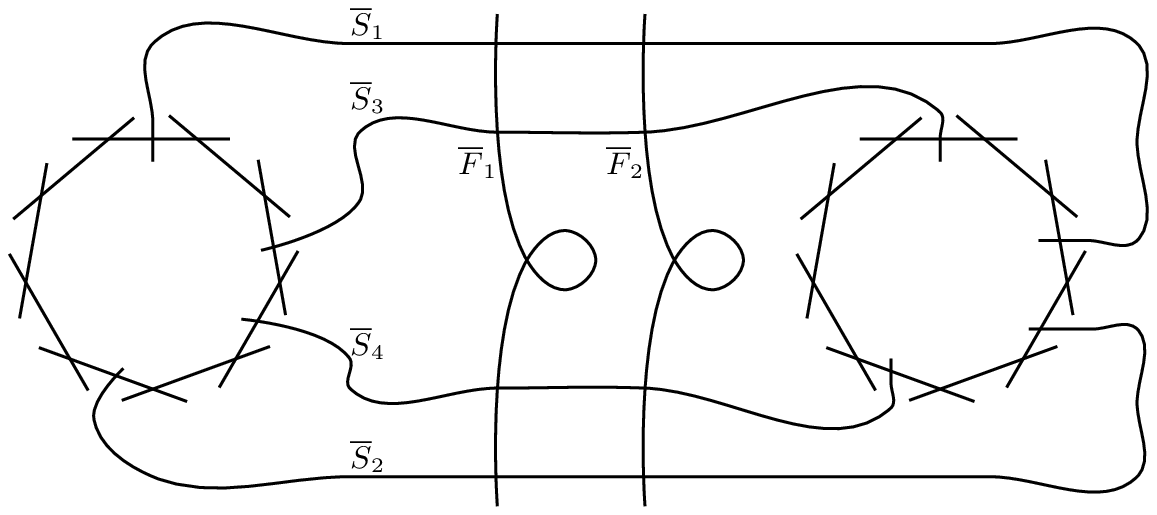}
\caption{A K3 surface $\overline{Y}$}
\label{figure:Ybar}
\end{figure}

\begin{figure}[hbtb]
\centering
\includegraphics{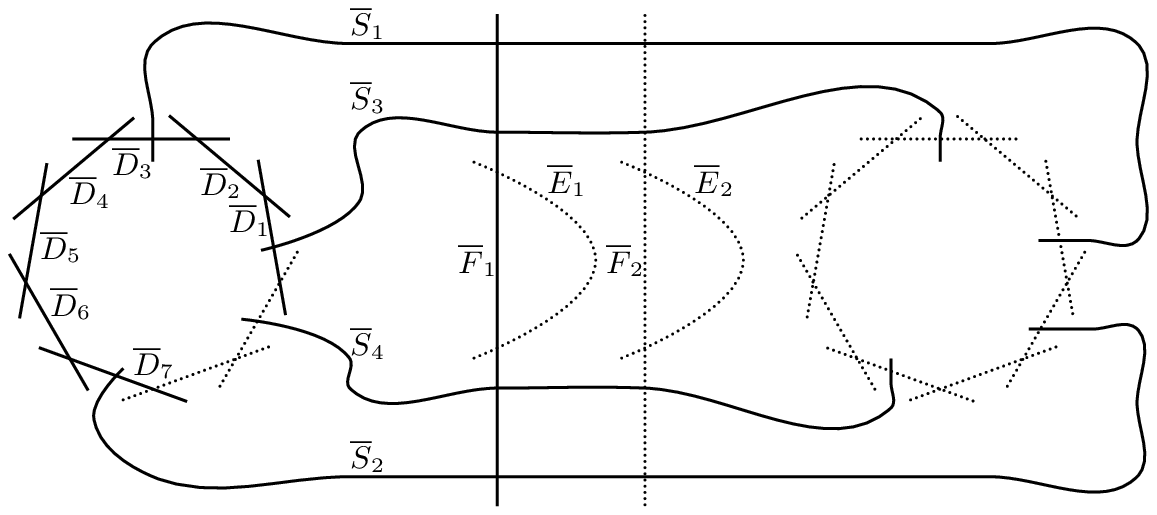}
\caption{A surface $\overline{V} = \blup{\overline{Y}}{2}$}
\label{figure:Vbar}
\end{figure}

The proof of Proposition~\ref{proposition:H^2=0} begins with the following two results.

\begin{proposition}[{Esnault and Viehweg~\cite[\S2]{Esnault-Viehweg}}]\label{proposition:Esnault-Viehweg}
Let $A = \sum_{i=1}^{r} {A_i}$ be a reduced normal crossing divisor on an algebraic manifold $W$. One has the exact sequence
\begin{equation*}
0 \to \Omega_W(\log{A}) \to \Omega_W(\log(A-A_1))(A_1) \to \Omega_{A_1}(\log(A-A_1)|_{A_1})(A_1) \to 0.
\end{equation*}
\end{proposition}

\begin{lemma}\label{lemma:H^0(Delta)=0}
$H^0(\Omega_{\overline{V}}(\log{\Delta})(K_{\overline{V}}))=0$.
\end{lemma}

\begin{proof}
Note that $K_{\overline{V}} = \overline{E}_1 + \overline{E}_2$. By Proposition~\ref{proposition:Esnault-Viehweg} we have an exact sequence
\begin{equation*}
0 \to \Omega_{\overline{V}}(\log(\Delta+\overline{E}_1))(\overline{E}_2) \to \Omega_{\overline{V}}(\log{\Delta})(\overline{E}_1+\overline{E}_2) \to \Omega_{\overline{E}_1}(\log{\Delta}|_{\overline{E}_1})(\overline{E}_1+\overline{E}_2) \to 0.
\end{equation*}
Since $\overline{E}_1^2 = -1$ and $\overline{E}_1\overline{E}_2=0$, we have $H^0(\Omega_{\overline{E}_1}(\log{\Delta}|_{\overline{E}_1})(\overline{E}_1+\overline{E}_2)) = 0$. Therefore it follows that
\begin{equation*}
H^0(\Omega_{\overline{V}}(\log{\Delta})(K_{\overline{V}})) = H^0(\Omega_{\overline{V}}(\log(\Delta + \overline{E}_1))(\overline{E}_2)).
\end{equation*}
Applying Proposition~\ref{proposition:Esnault-Viehweg} again, we have an exact sequence
\begin{equation*}
0 \to \Omega_{\overline{V}}(\log(\Delta+\overline{E}_1+\overline{E}_2)) \to \Omega_{\overline{V}}(\log(\Delta+\overline{E}_1))(\overline{E}_2) \to \Omega_{\overline{E}_2}(\log(\Delta+\overline{E}_1)|_{\overline{E}_2})(\overline{E}_2) \to 0.
\end{equation*}
Since $H^0(\Omega_{\overline{E}_2}(\log(\Delta+\overline{E}_1)|_{\overline{E}_2})(\overline{E}_2)) = 0$, we have
\begin{equation}\label{equation:log(Delta)(KV)=log(Delta+KV)}
\begin{split}
H^0(\Omega_{\overline{V}}(\log{\Delta})(K_{\overline{V}})) &= H^0(\Omega_{\overline{V}}(\log(\Delta + \overline{E}_1))(\overline{E}_2)) \\
&= H^0(\Omega_{\overline{V}}(\log(\Delta + \overline{E}_1+\overline{E}_2))).
\end{split}
\end{equation}
Therefore we need to show that $H^0(\Omega_{\overline{V}}(\log(\Delta + \overline{E}_1+\overline{E}_2))) = 0$.

We now consider the following exact sequence:
\begin{equation*}
0 \to \Omega_{\overline{V}} \to \Omega_{\overline{V}}(\log(\Delta + \overline{E}_1 + \overline{E}_2)) \to \bigoplus_{i=1}^{7} \sheaf{O_{\overline{D}_i}} \oplus \bigoplus_{i=1}^{4} \sheaf{O_{\overline{S}_i}} \oplus \sheaf{O_{\overline{F}_1}} \oplus \bigoplus_{i=1}^{2} \sheaf{O_{\overline{E}_i}} \to 0.
\end{equation*}
Note that the connecting homomorphism
\begin{equation*}
\delta: \bigoplus_{i=1}^{7} H^0(\sheaf{O_{\overline{D}_i}}) \oplus \bigoplus_{i=1}^{4} H^0(\sheaf{O_{\overline{S}_i}}) \oplus H^0(\sheaf{O_{\overline{F}_1}}) \oplus \bigoplus_{i=1}^{2} H^0(\sheaf{O_{\overline{E}_i}}) \to H^1(\Omega_{\overline{V}})
\end{equation*}
is the first Chern class map. Since the intersection matrix consisting of the intersection numbers of $\overline{D}_i$ ($i=1,\dotsc,7$), $\overline{S}_j$ ($j=1,\dotsc,4$), and $\overline{F}_1$, $\overline{E}_k$ ($k=1,2$) is invertible, their images by the map $\delta$ are linearly independent. Therefore the map $\delta$ is injective. Furthermore $H^0(\Omega_{\overline{V}}) = 0$. Hence, we have
\begin{equation*}
H^0(\Omega_{\overline{V}}(\log(\Delta + \overline{E}_1+\overline{E}_2)))=0.
\end{equation*}
Therefore it follows from \eqref{equation:log(Delta)(KV)=log(Delta+KV)} that
\begin{equation*}
H^0(\Omega_{\overline{V}}(\log{\Delta})(K_{\overline{V}}))  = H^0(\Omega_{\overline{V}}(\log(\Delta + \overline{E}_1+\overline{E}_2))) = 0. \qedhere
\end{equation*}
\end{proof}

\begin{proof}[Proof of Proposition~\ref{proposition:H^2=0}]
Since $K_{\overline{V}} = \psi^{\ast}{K_V}$, it follows from the projection formula that
\begin{equation*}
\psi_{\ast}(\Omega_{\overline{V}}(\log{\Delta})(K_{\overline{V}})) = \psi_{\ast}(\Omega_{\overline{V}}(\log{\Delta}))(K_V).
\end{equation*}
On the one hand, by the choice of $\Delta$, we have
\begin{equation*}
\Omega_{V}(\log{D}) \subset \psi_{\ast}(\Omega_{\overline{V}}(\log{\Delta})).
\end{equation*}
Therefore there is an injection
\begin{equation*}
0 \to \Omega_{V}(\log{D})(K_V) \to \psi_{\ast}(\Omega_{\overline{V}}(\log{\Delta})(K_{\overline{V}})).
\end{equation*}
Hence it follows by Lemma~\ref{lemma:H^0(Delta)=0} that
\begin{equation*}
H^0(V, \Omega_{V}(\log{D})(K_V)) \subset H^0(\overline{V}, \Omega_{\overline{V}}(\log{\Delta})(K_{\overline{V}})) = 0. \qedhere
\end{equation*}
\end{proof}

\section{A symplectic $4$-manifold with $b_2^+=1$, $K^2=5$, and $\pi_1 = \mathbb{Z}/2\mathbb{Z}$}

In this section we construct a symplectic $4$-manifold with $b_2^+=1$, $K^2=5$, and $\pi_1 = \mathbb{Z}/2\mathbb{Z}$ using a rational blow-down surgery.

We consider again the Enriques surface $Y$ used in Section~\ref{section:Construction}. We blow up five times totally at the five marked points $\bullet$. We blow up again three times and four times at the two marked points $\bigodot$, respectively. We then get a surface $Z = \blup{Y}{12}$; Figure~\ref{figure:Z5}. There exist two disjoint linear chains of $\mathbb{CP}^1$'s in $Z$:
\begin{align*}
&C_{4,1}: \uc{-6}-\uc{-2}-\uc{-2} \\
&C_{151,31}: \uc{-5}-\uc{-8}-\uc{-6}-\uc{-2}-\uc{-3}-\uc{-2}-\uc{-2}-\uc{-2}-\uc{-2}-\uc{-2}-\uc{-3}-\uc{-2}-\uc{-2}-\uc{-2}
\end{align*}

\begin{figure}[hbtb]
\centering
\includegraphics{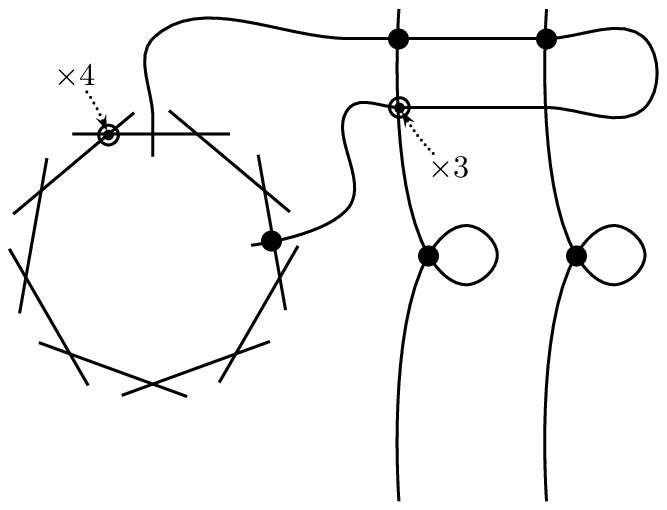}
\caption{An Enriques surface $Y$}
\label{figure:Y5}
\end{figure}

\begin{figure}[hbtb]
\centering
\includegraphics{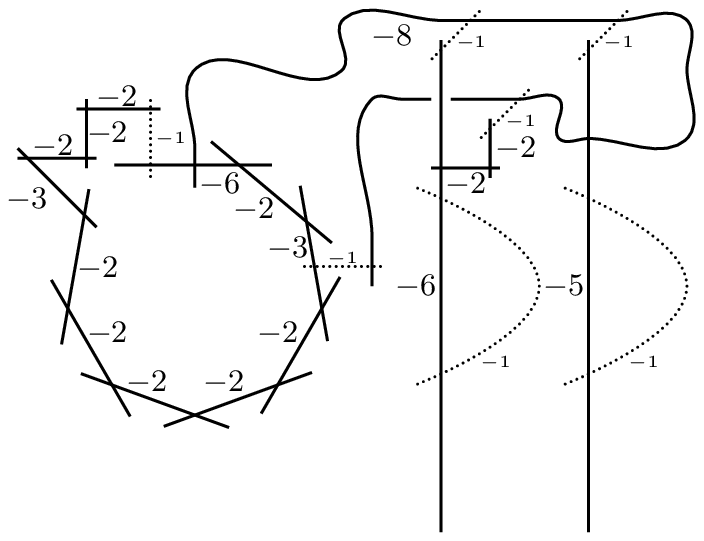}
\caption{A surface $Z = \blup{Y}{12}$}
\label{figure:Z5}
\end{figure}

We now perform a rational blow-down surgery of the surface $Z=\blup{Y}{12}$. By the results of Symington~\cite{Sym-1, Sym-2}, the rational blow-down $Z_{151,4}$ is a symplectic $4$-manifold. Thus we get a symplectic $4$-manifold $Z_{151,4}$ with $b_2^+=1$ and $K^2=5$. By applying a similar method in the proof of Proposition~\ref{proposition:pi_1=Z_2}, one can easily show that $\pi_1(Z_{151,4})=\mathbb{Z}/2\mathbb{Z}$.

\begin{theorem}
The rational blow-down $Z_{151,4}$ of the surface $Z$ in the construction above is a symplectic $4$-manifold with $b_2^+=1$, $K^2 =5$, and $\pi_1 = \mathbb{Z}/2\mathbb{Z}$.
\end{theorem}

\begin{remarks}
\begin{enumerate}[1.]
\item One can prove that the symplectic $4$-manifold $Z_{151,4}$ constructed in this section is minimal by using a technique in Ozsv\'{a}th and Szab\'{o}~\cite{Ozsvath-Szabo}.

\item It is an intriguing question whether the symplectic $4$-manifold $Z_{151,4}$ admit a complex structure. Since the cohomology $H^2 (T^0 _{X})$ is not zero in this case, it is hard to determine whether there exists a global $\mathbb{Q}$-Gorenstein smoothing. We leave this question for future research.

\item As a corollary, we can reconstruct a simply connected symplectic $4$-manifold with $b_2^+=3$ and $K^2=10$ from the symplectic $4$-manifold $Z_{151,4}$. We briefly sketch the construction. We consider the unramified double covering $\phi : \overline{Y} \to Y$ from the K3 surface $\overline{Y}$ to the Enriques surface $Y$ in the proof of Proposition~\ref{proposition:H^2=0}. Whenever we blow up $Y$ in the above construction of the surface $Z=\blup{Y}{12}$, we blow up twice $\overline{Y}$ at the preimages of $\pi$. Then we obtain an unramified double covering $\overline{\pi} : \overline{Z} \to Z$ from a complex surface $\overline{Z}=\blup{\overline{Y}}{24}$ which has four disjoint linear chains of $\mathbb{CP}^1$'s: two $C_{4,1}$'s and two $C_{151,31}$'s. We perform a rational blow-down surgery of the surface $\overline{Z}$. We then get a symplectic $4$-manifold $\overline{Z}_{151,4}$ with $b_2^+=3$ and $K^2=10$. Since $\pi_1(Z_{151,4}) = \mathbb{Z}/2\mathbb{Z}$ and there is an unramified double covering $\overline{Z}_{151,4} \to Z_{151,4}$ induced by the double covering $\pi : \overline{Z} \to Z$, the symplectic $4$-manifold $\overline{Z}_{151,4}$ is simply connected.
\end{enumerate}
\end{remarks}

\end{document}